\documentclass[12pt,a4paper]{amsart} 

\newtheorem{theorem}{Theorem}[equation] 
\newtheorem{lemma}[equation]{Lemma}     
\newtheorem{corollary}[equation]{Corollary}
\newtheorem{proposition}[equation]{Proposition}

\newcommand{\R}{\boldsymbol{R}}

\newcommand{\N}{\boldsymbol{N}}
\newcommand{\Z}{\boldsymbol{Z}}
\newcommand{\e}{\varepsilon}

\newcommand{\lbd}{\lambda}

\newcommand{\Schw}{\mathcal{S}}
\newcommand{\Ge}{\mathfrak{g}}
\newcommand{\Ges}{\mathfrak{g}^{\star}}
\newcommand{\Te}{\mathfrak{z}}
\newcommand{\T}{\mathfrak{z}^{\star}}
\renewcommand{\phi}{\varphi}
\newcommand{\op}{{\rm Op\,}}
\numberwithin{equation}{section}

\title[Invertibility on homogeneous groups]
{Invertibility of convolution operators on homogeneous groups} 
\author{P. G{\l}owacki}
\begin{document}
\maketitle

\begin{abstract}
We say that a tempered distribution $A$ belongs to the class $S^m(\Ge)$ on a homogeneous Lie algebra $\Ge$ if its Abelian Fourier transform $a=\widehat{A}$ is a smooth function on the dual $\Ges$ and satisfies the estimates
$$
|D^{\alpha}a(\xi)|\le C_{\alpha}(1+|\xi|)^{m-|\alpha|}.
$$
Let $A\in S^0(\Ge)$. Then the operator $f\mapsto f\star\widetilde{A}(x)$ is bounded  on $L^2(\Ge)$. Suppose that the operator is invertible and denote by $B$ the convolution kernel of its inverse. We show that $B$ belongs to the class $S^0(\Ge)$ as well. As a corollary we generalize Melin's theorem on the parametrix construction for Rockland operators.
\end{abstract}

\section*{Introduction}
In a former paper \cite{glowacki} we describe a calculus of a class of convolution operators on a nilpotent homogeneous group $G$ with the Lie algebra $\Ge$. These operators are distinguished by conditions imposed on the Abelian Fourier transforms of their kernels similar to those required from the $L^p$-multipliers on $\R^n$. More specifically, a tempered distribution $A$ belongs to the class $S^m(G)=S^m(\Ge)$ if its Fourier transform $a=\widehat{A}$ is a smooth function on the dual to the Lie algebra $\Ges$ and satisfies the estimates
$$
|D^{\alpha}a(\xi)|\le C_{\alpha}(1+|\xi|)^{m-|\alpha|},
\qquad
\xi\in\Ges.
$$
In \cite{glowacki} we follow and extend to the setting of a general homogeneous group the ideas of Melin \cite{melin} who first introduced such a calculus on the subclass of stratified groups. The classes $S^m(\Ge)$ of symbols of convolution operators have the expected properties of composition and boundedness (see Propositions \ref{calculus} and \ref{cv} below) which is a generalization of the results of Melin \cite{melin}. However, a complete calculus should also deal with the problem of invertibility. The aim of the present paper is to fill the gap.

Suppose that $A\in S^0(\Ge)$. Then, by the boundedness theorem (see Proposition \ref{cv} below), the operator
$$
f\mapsto f\star\widetilde{A}(x)=\int_{\Ge}f(xy)A(y)\,dy
$$
defined initially on the Schwartz class functions extends uniquely to a bounded operator on $L^2(\Ge)$. Furthermore, suppose that the operator  $f\mapsto f\star\widetilde{A}$ is invertible on $L^2(\Ge)$ and denote by $B$ the convolution kernel of its inverse. We show here that under these circumstances $B$ belongs to the class $S^0(\Ge)$ as well. This is done by replacing Melin's techniques of parametrix construction involving the more refined classes $S^{m,s}(\Ge)\subset S^m(\Ge)$  of convolution operators by the calculus of less restrictive classes $S_0^m(\Ge)$, where no estimates in the central directions are required. 

Let us remark that the described result can be also looked upon as a close analogue of the theorem on the inversion of singular integrals, see \cite{glowacki2} and Christ-Geller \cite{christ-geller}. 

By using auxiliary convolution operators, namely accretive homogeneous kernels $P^m$ smooth away from the origin, we construct "elliptic" operators $V_1^m$ of order $m>0$ and get inversion  results for classes $S^m(\Ge)$ for all $m>0$, which enables us to generalize Melin's theorem on the parametrix construction for Rockland operators. At the same time, however, we present a direct parametrix construction for Rockland operators which avoids the machinery of Melin and also that of the present paper and depends only on well-known properties of Rockland operators as derived in Folland-Stein \cite{folland} and the calculus of \cite{glowacki}. 

We believe that the presented symbolic calculus may be a step towards a more comprehensive pseudodifferential calculus on nilpotent Lie groups parallel to that of Christ-Geller-G{\l}owacki-Polin \cite{cggp}. 

\section{Symbolic calculus.}

Let $\Ge$ be a nilpotent Lie algebra endowed with a family of dilations $\{\delta_t\}_{t>0}$. We identify $\Ge$ with the corresponding nilpotent Lie group by means of the exponential map. Let  
$$
1=p_1<p_2<\dots<p_d
$$ 
be the exponents of homogeneity of the dilations. Let 
$$
\Ge_j=\{x\in\Ge: tx=t^{p_j}\cdot x\},
\hspace{2em}
1\le j\le d.
$$
Denote by $Q=\sum_k\dim\Ge_k\cdot p_k$ the homogeneous dimension of $\Ge$. Let $|\cdot|$ be a homogeneous norm on $\Ge$ and $\rho$ be a smooth function on $\Ge$ such that
$$
c(1+|x|)\le \rho(x)\le C(1+|x|),
\qquad
x\in\Ge,
$$
for some $C\ge c>0$. A similar notation will be applied to the dual space $\Ges$. 

In expressions like $D^{\alpha}$ or $x^{\alpha}$ we shall use multiindices
$$
\alpha=(\alpha_1,\alpha_2,\dots,\alpha_d),
$$
where 
$$
\alpha_k=(\alpha_{k1},\alpha_{k1},\dots,\alpha_{kn_k}),
$$ 
are themselves multiindices with positive integer entries corresponding to the spaces $\Ge_k$ or $\Ges_k$. The homogeneous length of $\alpha$ is defined by
$$
|\alpha|=\sum_{k=1}^d|\alpha_k|,
\qquad|\alpha_k|=\dim\Ge_k\cdot p_k. 
$$

As usual we denote by $\Schw(\Ge)$ or $\Schw(\Ges)$ the Schwartz classes of smooth and rapidly vanishing functions. The Fourier transform 
$$
\widehat{f}(\xi)=\int_{\Ge}f(x)\,e^{-i\langle\xi,x\rangle}\,dx
$$
maps $\Schw(\Ge)$ onto $\Schw(\Ges)$ and extends to tempered distributions on $\Ge$. Let
$$
\|f\|^2=\int_{\Ge}|f(x)|^2\,dx,
\qquad
f\in L^2(\Ge).
$$
A similar notation will be applied to $f\in L^2(\Ges)$, where the Lebesgue measure $d\xi$ on $\Ges$ is normalized so that 
$$
\int_{\Ge}|f(x)|^2\,dx=\int_{\Ges}|\widehat{f}(x)|^2\,d\xi.
$$
The algebra of bounded linear operators on $L^2(\Ge)$ will be denoted by $\mathcal{B}(L^2(\Ge))$.

For a tempered distribution $A$ on $\Ge$, we write 
$$
{\rm Op}(A)f(x)=f\star\widetilde{A}(x)=\int_{\Ge}f(xy)A(dy),
\hspace{2em}
f\in\Schw(\Ge).
$$

Let $m\in\R$. By $S^m(\Ge)=S^m(\Ge,\rho)$ we denote the class of all distributions $A\in\Schw'(\Ge)$ whose Fourier transforms $a=\widehat{A}$ are smooth and satisfy the estimates
\begin{equation}\label{symbolic_estimates}
|D^{\alpha}a(\xi)|\le C_{\alpha}\rho(\xi)^{m-|\alpha|},
\end{equation}
where $|\alpha|$ stands for the homogeneous length of a multiindex. Let us recall that $S^m(\Ge)$ is a Fr\'echet space with the family of seminorms 
$$
|a|_{\alpha}=\sup_{\xi\in\Ges}\rho(\xi)^{-m+|\alpha|}|D^{\alpha}a(\xi)|.
$$
It is not hard to see that for every $\phi\in C_c^{\infty}(\Ge)$ equal to $1$ in a neighbourhood of $0$ the distribution $(1-\phi)A$ is a Schwartz class function. Thus
\begin{equation}\label{split}
A=A_1+F, 
\end{equation}
where $A_1$ is compactly supported and $F\in\Schw(\Ge)$.

It follows from (\ref{split})  that for every $m\in\R$
$$
{\rm Op}(A):\Schw(\Ge)\to\Schw(\Ge)
$$
is a continuous mapping if $A\in S^m(\Ge)$. Therefore, it extends to a continuous mapping, denoted by the same symbol, of $\Schw'(\Ge)$.  It is also clear that for $A\in S^m(\Ge)$ and $B\in S^n(\Ge)$ the convolution $A\star B$ makes sense and $\op(A\star B)=\op(A)\op(B)$.

The following two propositions have been proved in \cite{glowacki}.  

\begin{proposition}\label{calculus}
If $A\in S^m(\Ge)$ and $B\in S^n(\Ge)$, then $A\star B\in S^{m+n}(\Ge)$ and the mapping
$$
S^m(\Ge)\times S^n(\Ge)\ni(A,B)\mapsto A\star B\in S^{m+n}(\Ge)
$$
is continuous.
\end{proposition}

\begin{proposition}\label{cv}
If $A\in S^0(\Ge)$, then ${\rm Op}(A)$ is bounded on $L^2(\Ge)$ and the mapping
$$
S^0(\Ge)\ni A\mapsto{\rm Op}(A)\in\mathcal{B}(L^2(\Ge))
$$
is continuous.
\end{proposition}

Let  $\Te$ be the central subalgebra corresponding to the largest eigenvalue of the dilations. We may assume that
\begin{equation}\label{central}
\Ge=\Ge_0\times\Te,
\hspace{2em}
\Ges=\Ges_0\times\T,
\end{equation}
where $\Ge_0$ may be identified with the quotient Lie algebra $\Ge/\Te$. The multiplication law in $\Ge$ can be expressed by
$$
(x,t)(y,s)=(x\circ y, t+s+r(x,y)),
$$
where $x\circ y$ is mutiplication in $\Ge_0$. Here the variable in $\Ge$ has been split in accordance with the given decomposition. In a similar way we also split the variable $\xi=(\eta,\lbd)$ in $\Ges$. 

Let $m\in\R$. By $S_0^m(\Ges)$ we denote the class of all distributions $A\in\Schw'(\Ge)$ whose Fourier transforms $a=\widehat{A}$ are smooth in the variable $\eta$ and satisfy the estimates  
\begin{equation}\label{symbolik_estimates}
|D_{\eta}^{\alpha}a(\eta,\lbd)|\le C_{\alpha}\rho(\eta,\lbd)^{m-|\alpha|}.
\end{equation}
Again, $S_0^m(\Ge)$ is a Fr\'echet space with the family of seminorms 
$$
|a|_{\alpha}=\sup_{(\eta,\lbd)\in\Ges}\rho(\eta,\lbd)^{-m+|\alpha|}|D_{\eta}^{\alpha}a(\eta,\lbd)|.
$$

The following result has not been stated explicitly in \cite{glowacki} but follows by the argument given there.

\begin{proposition}\label{calculus_omega}
If $A\in S_0^m(\Ges)$ and $B\in S_0^n(\Ges)$, then $A\star B\in S_0^{m+n}(\Ges)$ and the mapping
$$
S_0^m(\Ges)\times S_0^n(\Ges)\ni(A,B)\mapsto A\star B\in S_0^{m+n}(\Ges)
$$
is continuous.
\end{proposition}

Let us introduce the following notation:
$$
\widehat{f}\#\widehat{g}(\xi)=\widehat{f\star g}(\xi),
\qquad
\xi\in\Ges,
$$
for $f,g\in\Schw(\Ge)$. Then, for every fixed $\lbd\in\Te^{\star}$,
\begin{equation}\label{lambda}
a\#b(\eta,\lbd)=a(\cdot,\lbd)\#_{\lbd}b(\cdot,\lbd)(\eta),
\end{equation}
where 
$$
\widehat{f}\#_{\lbd}\widehat{g}(\eta)=(f\star_{\lbd}g)\ \widehat{}\ (\eta),
\hspace{2em}
f\star_{\lbd}g(x)=\int_{\Ge_0}f(x\circ y^{-1})g(y)\,e^{i\langle r(x,y^{-1}),\lbd\rangle}dy
$$
for $f,g\in\Schw(\Ge_0)$. In particular, $f\star_0g$ is the usual convolution on the quotient group~$\Ge_0$.

Let
$$
T_{k_j}F(x)=ix_{k_j}F(x),
\qquad
T_{\alpha}F(x)=(ix)^{\alpha}F(x).
$$
For  a multiindex $\gamma$, let 
$$
k(\gamma)=\max_{1\le k\le d}\{k:\gamma_k\neq0\},
\qquad
\gamma\neq0,
$$
and $k(0)=0$.
\begin{lemma}\label{commutator}
Let $f,g\in\Schw(\Ge)$. Then for every $\gamma$,
\begin{equation}\label{leibniz}
T_{\gamma}(f\star g)=T_{\gamma}f\star g+f\star T_{\gamma}g
+\sum_{\alpha,\beta}
c_{\alpha\beta}^{\gamma}T_{\alpha}f\star T_{\beta}g
\end{equation}
or, equivalently, by applying the Fourier transform, 
\begin{equation}\label{Leibniz}
D^{\gamma}(f\#g)=D^{\gamma}f\#g+f\# D^{\gamma}g
+\sum_{\alpha,\beta}
c_{\alpha\beta}^{\gamma}\,D^{\alpha}f\# D^{\beta}g,
\end{equation}
where the summation extends over 
\begin{equation}\label{summa}
k(\alpha),k(\beta)\le k(\gamma), \ |\alpha|+|\beta|=|\gamma|, \  \alpha\neq0, \beta\neq0.
\end{equation}
\end{lemma}
\begin{proof}[Proof sketch.]
This is proved by the following induction. We pick a $\gamma$ and assume that (\ref{leibniz}) holds for all $\gamma'\neq\gamma$  such that $k(\gamma')\le k(\gamma)$ and $|\gamma'|\le|\gamma|$. By (1.22) of Folland-Stein \cite{folland}, the group law is expressed by
$$
(xy^{-1})_k=x_k-y_k+P_k(x,y),
\qquad
1\le k\le d,
$$
where the polynomials $P_k$ depend only on variables $x_j,y_j$ with $j<k$. Then
\begin{equation}\label{ind}
T_{x_{k_j}}(f\star g)=\sum_{\alpha,\beta}T_{\alpha}(f\star T_{\beta}g),
\end{equation}
where $k(\alpha),k(\beta)<k$ and $|\alpha|+|\beta|=p_k$. We let $T_{\gamma}=T_{x_{k_j}}T_{\gamma'}$ and apply the induction hypothesis combined with $(\ref{ind})$. Details are left to the reader. 
\end{proof}

\begin{lemma}\label{s0}
Let $A\in S^m(\Ge)$. If $B\in S_0^{-m}(\Ge)$ is the inverse of $A$, that is,
$$
A\star B=B\star A=\delta_0,
$$
then $B\in S^m(\Ge)$. 
\end{lemma}

\begin{proof}
Let $a=\widehat{A}$, $b=\widehat{B}$. By (\ref{Leibniz}),
$$
0=D^{\gamma_d}(a\#b)=D^{\gamma_d}a\#b+a\#D^{\gamma_d}b
+\sum c_{\alpha\beta}^{\gamma}D^{\alpha}a\#D^{\beta}b,
$$
where the summation extends over $\alpha,\beta$ such that
$$
|\alpha|+|\beta|=|\gamma_d|, \ |\alpha_d|,|\beta_d|<|\gamma_d|
$$
and every multiindex is split as $\alpha=(\alpha',\alpha_d)$, $\alpha_d$ being the part corresponding to $\Ges_d$. 
Therefore,
$$
D^{\gamma_d}b=-b\#D^{\gamma_d}a\#b
+\sum c_{\alpha\beta}^{\gamma}b\#D^{\alpha}a\#D^{\beta}b,
$$
where the sybol on the right-hand side belongs to $\widehat{S}_0^{-m-|\gamma_d|}$ provided that $b\in\widehat{S}_0^{-m-\kappa}$ for $|\kappa|<|\gamma_d|$. By induction, $D^{\gamma_d}b\in\widehat{S}_0^{-m-|\gamma_d|}(\Ge)$, which is our assertion. 
\end{proof}

Let $A_j\in S_0^{m_j}(\Ges)$, where $m_j\searrow-\infty$. Then there exists a distribution $A\in S_0^{m_1}(\Ges)$ such that
$$
A-\sum_{j=1}^NA_j\in S_0^{m_{N+1}}(\Ges)
$$
for every $N\in\N$. The distribution $A$ is unique modulo the class 
$$
S_0^{-\infty}(\Ges)=\bigcap_{n<0}S_0^n(\Ges).
$$
  We shall write
\begin{equation}\label{asymptotic}
A\approx\sum_{j=1}^{\infty}A_j,
\end{equation}
and call the distribution $A$ the asymptotic sum of the series $\sum A_j$ (cf., e.g., H\"ormander \cite{hormander}, Proposition 18.1.3).

We say that $A\in S^m(\Ge)$, where $m\ge0$, has a parametrix $B\in S^{-m}(\Ge)$ if
$$
B\star A-\delta_0\in \Schw(\Ge),
\qquad
A\star B-\delta_0\in \Schw(\Ge),
$$
where $\delta_0$ stands for the Dirac delta at $0$. If $B_1$ is a left-parametrix and $B_2$ a right one, then it is easy to see that $B_1=B_2$ modulo the Schwartz class functions so both $B_1$ and $B_2$ are parametrices. In particular, if $A$ is symmetric, then either of the conditions implies the other one.

\section{Sobolev spaces}

We  say that a tempered distribution $T$ is a {\it regular kernel of order $r\in\R$}, if it is homogeneous of degree ${}-Q-r$ and smooth away from the origin. A symmetric distribution $T$ is said to be {\it accretive}, if
$$
\langle T,f\rangle\ge0
$$
for real $f\in C_c^{\infty}(\Ge)$ which attain their maximal value at $0$. Such a $T$ is an infinitesimal generator of a continuous semigroup of subprobability measures $\mu_t$. By the Hunt theory (see, eg., Duflo \cite{duflo}), $\op(T)$ is a positive selfadjoint operator on $L^2(\Ge)$ with $\Schw(\Ge)$ as its core domain, and for every $0<m<1$,
$$
\op(T)^m=\op(T^m),
\qquad
\langle T^m,f\rangle=\frac{1}{\Gamma(-m)}\int_0^{\infty}t^{-1-m}\langle
\delta_0-\mu_t, f\rangle\,dt,
$$
where the distribution $T^m$ is also accretive.

Let $T$ be a fixed symmetric accretive regular kernel of order $0<m\le1$. Then there exists a symmetric nonnegative function $\Omega\in C^{\infty}(\Ge\setminus\{0\})$ which is homogeneous of degree $0$ such that
$$
\langle T,f\rangle
=cf(0)+\lim_{\e\to}\int_{|x|\ge\e}\Big(f(0)-f(x)\Big)\frac{\Omega(x)\,dx}{|x|^{Q+m}},
$$
where $c\ge0$. If $c=0$, $T$ is an infinitesimal generator of a continuous semigroup of probability measures with smooth densities. For every $0<a<1$, $T^a$ is also a symmetric regular kernel of order $am$. 

Let 
$$
\langle P,f\rangle
=\lim_{\e\to}\int_{|x|\ge\e}\frac{f(0)-f(x))\,dx}{|x|^{Q+1}}
$$
be a fixed symmetric accretive distribution of order $1$. Let us warn the reader that the distributions $P^m$ do not belong to any of the classes $S^m(\Ge)$ as they do not vanish rapidly at infinity which is a certain technical complication. That is why we introduce  the truncated kernels
$$
V_0=I,
\qquad
V_m=\phi P^m,
\qquad
m>0,
$$ 
where $\phi$ is a symmetric nonnegative $[0,1]$-valued smooth function with compact support and equal to $1$ on the unit ball. Thus defined $V_m\in S^m(\Ge)$ is also accretive and differs from $P^m$ by a finite measure. Therefore, for every $0<m\le 1$, there exist constants $C_1>0$ and $C_2>0$ such that
\begin{equation}\label{equi}
C_1\|(I+\op(P))^mf\|\le\|(I+\op(V_m))f\|\le C_2\|(I+\op(P))^mf\|,
\end{equation}
for $f\in\Schw(\Ge)$.

\begin{proposition}\label{ge}
For every $0<m\le1$, there exists a constant $C_m>0$ such that
$$
\|f\star V_m\|\ge C_m\|f\|,
\hspace{2em}
f\in\Schw(\Ge).
$$
\end{proposition}

\begin{proof}
In fact, let $f\in\Schw(\Ge)$ and $F=\widetilde{f}\star f$. Then
\begin{equation*}\begin{split}
\langle f\star V_m,f\rangle&=\langle T,F\rangle\\
&=\lim_{\e\to}
\int_{\e\le|x|\le1}\Big(F(0)-\phi(x)F(x)\Big)
\frac{\Omega_m(x)\,dx}{|x|^{Q+1}}+F(0)\int_{|x|\ge1}\frac{\Omega_m(x)\,dx}{|x|^{Q+1}}\\
&\ge C_m^2F(0)=C_m^2\|f\|^2
\end{split}\end{equation*}
since the first integral is nonnegative.
\end{proof}

It follows from (\ref{equi}) and Proposition \ref{ge} that there exist new constants  $C_1>0$ and $C_2>0$ such that
\begin{equation}\label{equi1}
C_1\|(I+\op(P))^mf\|\le\|\op(V_m)f\|\le C_2\|(I+\op(P))^mf\|,
\end{equation}
for $f\in\Schw(\Ge)$ and $0\le m\le1$.

Recall from \cite{glowacki1} that $P$ is {\it maximal}, that is, for every regular symmetric kernel $T$ of arbitrary order $m>0$ there exists a constant $C>0$ such that
\begin{equation}\label{max}
\|f\star\widetilde{T}\|\le C\|f\star P^mf\|,
\hspace{2em}
f\in\Schw(\Ge).
\end{equation}

We introduce a scale of Sobolev spaces. For every $m\in\R$,
$$
H(m)=\{f\in L^2(\Ge): (I+\op(P))^mf\in L^2(\Ge)\}
$$
with the usual norm $\|f\|_{(m)}=\|(I+\op(P))^mf\|_2$. The dual space to $H(m)$ can be identified with $H(-m)$. By (\ref{equi1}), the norms defined by $V_m$ for $0<m\le1$ are equivalent. It follows that for every $0\le m\le1$,
$$
V_m:H(m)\to H(0)
$$
is an isomorphism.

\section{Main step}

Here comes a preliminary version of our theorem.
\begin{proposition}\label{f}
Let $0\le m\le1$. Let  $A=A^{\star}\in S^m(\Ge)$ and let $\op(A):H(m)\to H(0)$ be an isomorphism. If $A\star V_m=V_m\star A$, then there exists $B\in S^{-m}(\Ge)$ such that
$$
A\star B=B\star A=\delta_0.
$$
In particular $\op(B)=\op(A)^{-1}$.
\end{proposition}

By hypothesis, $A$ is invertible in $\mathcal{B}(L^2(\Ge))$. There exists a symmetric distribution $B$ such that
$$
\op(A)^{-1}f=f\star{B},
\qquad
f\in\Schw(\Ge).
$$
We have to show that $B\in S^{-m}(\Ge)$.

Let $\Schw_1(\Ge)$ denote the subspace of $\Schw(\Ge)$ consisting  of those functions whose Fourier transform is supported where $1\le|\lbd|\le2$. Note that this subspace is invariant under convolutions.

\begin{lemma}\label{z}
$\op(B)$ maps continuously $\mathcal{S}(\Ge)$ into $\mathcal{S}(\Ge)$. The same applies to the invariant space 
$\Schw_1(\Ge)$.
\end{lemma}

\begin{proof}[Proof sketch.]
Being a convolution operator bounded on $L^2(\Ge)$, $\op(B)$ commutes with right-invariant vector fields $Y$, hence it maps $\Schw(\Ge)$ into $L^2(\Ge)\cap C^{\infty}(\Ge)$. Therefore, it is sufficient to show that for every $\gamma$, $T_{\gamma}\op(B)$ is bounded in $L^2$-norm. By Lemma \ref{commutator},
\begin{equation}\label{induction}
 \begin{split}
T_{\gamma}\op(B)&=\op(B)T_{\gamma}+\op(B)[T_{\gamma},\op(A)]\op(B)\\
&=\op(B)T_{\gamma}+\op(B)\op(A_{\gamma})\op(B)\\
&+\sum_{\alpha,\beta}
c_{\alpha\beta}\cdot\op(B)\op(A_{\alpha})T_{\beta}\op(B),
 \end{split}
\end{equation}
where the summation is taken over $\alpha$, $\beta$ as in (\ref{summa}).
The proof is completed by induction very similar to that of the proof of Lemma \ref{commutator}.
\end{proof}

For $n\in\Z$, let
$$
\langle A_n,f\rangle=2^{-nm}\int_{\Ge}f(2^nx)\,A(dx),
\qquad
\langle B_n,f\rangle=2^{nm}\int_{\Ge}f(2^nx)\,B(dx).
$$

\begin{corollary}\label{e}
The operators $\op(B_n)$ are equicontinuous on $\Schw_1(\Ge)$.
\end{corollary}

\begin{proof}
Let $\phi\in C_c^{\infty}(\Ges_d)$ be equal to $1$ for $1\le|\lbd|\le2$. Then $\op(A_n)f=\op(A_n')f$, $\op(B_n)f=\op(B_n')f$ for $f\in\Schw_1(\Ge)$, where
\[
 \widehat{A_n'}(\eta,\lbd)=\widehat{A_n}(\eta,\lbd)\phi(\lbd),
\qquad
 \widehat{B_n'}(\eta,\lbd)=\widehat{B_n}(\eta,\lbd)\phi(\lbd).
\]

By Proposition \ref{cv}, the mapping
$$
S^m(\Ge)\ni A\to \op(B)\in\mathcal{B}(L^2(\Ge))
$$
is continuous. Since the family $\{A_n'\}$ is bounded in $S^m(\Ge)$ so is $\{\op(B_n')\}$ in $\mathcal{B}(L^2(\Ge))$. Hence our assertion follows by induction using (\ref{induction}).
\end{proof}

Let $a=\widehat{A}$, and let
$$
\widehat{A_{\lbd}}(\eta)=a_{\lbd}(\eta)=a(\eta,\lbd),
\qquad
\lbd\in\T.
$$

\begin{lemma}\label{conti}
For every $f\in\Schw(\Ges_0)$ the function
$$
\lbd\to\|f\#_{\lbd}a_{\lbd}\|^2
$$
is continuous.
\end{lemma}

\begin{proof}
Let $0<h\in\Schw(\T)$ and $h(0)=1$. Then $F=(f\otimes h)\#a\in\Schw(\Ges)$ and
$$
\lbd\to\int_{\Ges_0}|F(\eta,\lbd)|^2d\eta=|h(\lbd)|^2\|f\#_{\lbd}a_{\lbd}\|^2
$$
is continuous, which implies our claim.
\end{proof}

From now on we  proceed by induction. The assertion of Proposition \ref{f} is obviously true in the Abelian case. Let us assume that it holds for $\Ge_0=\Ge\slash\Te$. 

\begin{lemma}\label{c}
The distribution $A_0$ satisfies the hypothesis of  Proposition \ref{f} on $\Ge_0$.
\end{lemma}

\begin{proof}
Observe that under the remaining assumptions of Proposition \ref{f} the condition that $\op(A):H(m)\to H(0)$ is an isomorphism is equivalent to the estimate
$$
\|f\star A\|\ge C\|f\star V_m\|,
\qquad
f\in \Schw(\Ge).
$$

Now, since $A\star V_m=V_m\star A$, we also have
$$
A_0\star(V_m)_0=(V_m)_0\star A_0,
$$
where $(V_m)_0$ is the counterpart of $V_m$ on $\Ge_0$.
Furthermore, we have 
$$
\|f\star {A}\|\ge C\|f\star V_m\|
$$
so, by Lemma \ref{conti},
$$
\|f_0\star {A_0}\|\ge C\|f_0\star(V_m)_0\|,
\qquad
f\in\Schw(\Ge),
$$
which implies
$$
\|f\star {A_0}\|\ge C\|f\star(V_m)_0\|,
\qquad
f\in\Schw(\Ge_0).
$$
\end{proof}

Let $b=\widehat{B}$ and $b_n=\widehat{B_n}$. Of course, $b_n\in\Schw'(\Ges)$.

\begin{lemma}
There exist $p\in\widehat{S}_0^{-m}(\Ges)$ and $q\in\Schw(\Ges)$ such that
\begin{equation}\label{d}
p\#a(\eta,\lbd)=1-q(\eta,\lbd),
\qquad
|\lbd|\le2.
\end{equation}
\end{lemma}

\begin{proof}
Let $u\in C_c^{\infty}([0,\infty)$ be equal to $1$ in a neighbourhood of $[0,1]$ and supported in $[0,2)$. Then
$$
\psi(\eta,\lbd)=u\Big(\frac{\rho(0,\lbd)}{\rho(\eta,0)}\Big)
$$
is an element of $\widehat{S}^{0}(\Ges)$. By Lemma \ref{c} and the induction hypothesis, there exists $b_0\in\widehat{S}^{-m}(\Ges)$ on $a$ such that
$$
b_0\#_0a_0=1.
$$
Let
$$
p(\eta,\lbd)=\psi(\eta,\lbd)b_0(\eta).
$$
Then $p\in\widehat{S}^{-m}(\Ges)$ and
\begin{equation*}
 \begin{split}
p\#a(\eta,\lbd)&=p\#(a-a_0)(\eta,\lbd)+b_0\#_0a_0(\eta)+(1-\psi)(\cdot,\lbd)b_0\#_0a_0(\eta)\\
&=1-q_0(\eta,\lbd),
\end{split}
\end{equation*}
where for every $\phi\in C_c^{\infty}(\T)$, $\phi(\lbd)q_0(\eta,\lbd)$ is in $\widehat{S}_0^{-1}(\Ges)$. Therefore we take $\phi\in C_c^{\infty}(\T)$ which equals $1$ where $|\lbd|\le2$ and modify $p_0$ and $q_0$ by letting
$$
p_1(\eta,\lbd)=p_0(\eta,\lbd)\phi(\lbd),
\qquad
q_1(\eta,\lbd)=q_0(\eta,\lbd)\phi(\lbd).
$$
Now, $p_1\in\widehat{S}_0^{-m}(\Ges)$, $q_1\in\widehat{S}_0^{-1}(\Ges)$, and
$$
p_1\#a=1-q_1,
\qquad
|\lbd|\le2.
$$

Let
$$
p\approx\sum_{k=1}^{\infty}q_1^k\#p_1,
$$
where the infinite sum is understood as in (\ref{asymptotic}). Then $p\in S_0^{-m}$ and
$$
p\#a=1-q,
\qquad
|\lbd|\le2,
$$
where $q\in\Schw(\Ges)$.
\end{proof}

\vspace{2ex}
Now we are in a position to conclude the proof of Proposition \ref{f}. By acting with $b$ on the right on both sides of (\ref{d}), we get
$$
b=p+q\#b,
\qquad|\lbd|\le2,
$$
where $q\#b\in\Schw(\Ge)$ (Lemma \ref{z}). Consequently,
$$
|D^{\alpha}_{\eta}b(\eta,\lbd)|\le C_{\alpha}\rho(\eta,\lbd)^{-m-|\alpha|},
\qquad
|\lbd|\le2.
$$
However, the same applies to $b_n$ for every $n\in\N$ with the same constants $C_{\alpha}$. Therefore, by (\ref{symbolik_estimates}), $B\in S^{-m}_0(\Ge)$.
Finally, by Lemma \ref{s0}, we conclude that $B\in S^{-m}(\Ge)$.

\bigskip
\begin{corollary}\label{fe}
Let  $A\in S^0(\Ge)$ and let 
$$
\|f\star A\|\ge C\|f\|,
\qquad
f\in\Schw(\Ge).
$$
There exists $B\in S^{0}(\Ge)$ such that
$$
B\star A=\delta_0.
$$
\end{corollary}

\begin{proof}
In fact,
\begin{align*}
C^2\|f\|^2\le\|\op(A)f\|^2&=<\op(A^{\star}\star A)f,f>\\
&\le\|\op(A^{\star}\star A)f\|\|f\|,
\end{align*}
for $f\in\Schw(\Ge)$ so $\op(A^{\star}\star A):L^2(\Ge)\to L^2(\Ge)$ is an isomorphism. By Proposition \ref{f} there exists $B_1\in S^0(\Ge)$ such that $B_1\star A^{\star}\star A=\delta_0$. Therefore $B_1\star A^{\star}$ is the  left-inverse for $A$. 
\end{proof}

\begin{corollary}
For every $0\le m\le1$, there exists $V_{-m}\in S^{-m}(\Ge)$ such that
$$
V_m\star V_{-m}=V_{-m}\star V_m=\delta_0.
$$ 
\end{corollary}

\bigskip
\section{The operator $\op(V_1)$}

In this section we show that the role of the family of distributions $V_m\in S^m(\Ge)$ in defining the Sobolev spaces can be taken over by the family of fractional powers of one single distribution $V_1$. This will enable the final step towards our theorem.

Recall that if a positive selfadjoint operator $A:L^2(\Ge)\to L^2(\Ge)$ is invertible, then
\begin{equation}\label{powers}
A^{-k}f=\frac{\sin k\pi}{\pi}\int_0^{\infty}t^{-k}(tI+A)^{-1}f\,dt
\end{equation}
for $0<k<1$ (see, e.g, Yosida \cite{yosida}, IX.11).

The operator $\op(V_1)$ is positive selfadjoint and invertible. In the proof of the next proposition we follow Beals \cite{beals1}, Theorem 4.9.
\begin{proposition}
For every $m\in\R$, $\op(V_1)^m=\op(V_1^m)$, where $V_1^m\in S^m(\Ge)$.
\end{proposition}

\begin{proof}
It is sufficient to prove the proposition for $-1<m<0$. For $t\ge0$ let
$$
R_t=(V_1+t\delta_0)^{-1},
\hspace{2em}
r_t=\widehat{R}_t.
$$
The operators $\op(V_1)+tI$ satisfy the hypothesis of Proposition \ref{f} with the exponent $m=1$ uniformly so there exist constants $C'_{\alpha}$ independent of $t$ such that
\begin{equation}\label{res}
|D^{\alpha}r_t|\le C'_{\alpha}\rho^{-1-|\alpha|}.
\end{equation}
On the other hand
$$
tR_t=\delta_0-R_t\star V_1\in S^0(\Ge)
$$
uniformly in $t$ so that
\begin{equation}\label{resolv}
t|D^{\alpha}r_t|\le C''_{\alpha}\rho^{-\alpha}.
\end{equation}
Combining (\ref{res}) with (\ref{resolv}) we get
$$
|D^{\alpha}r_t|\le C_{\alpha}(t+\rho)^{-1}\rho^{-\alpha}
$$
with $C_{\alpha}$ independent of $t\ge0$.

Now, the operator $\op(V_1)$ is positive and invertible so, by (\ref{powers}), $\op(V_1)^m=\op(V_1^m)$, where
$$
(V_1^{m})^{\wedge}=-\frac{\sin m\pi}{\pi}\int_0^{\infty}t^{m}r_t\,dt,
$$
where $-1<m<0$. Therefore
\begin{equation*}\begin{split}
|D^{\alpha}(V_1^{m})^{\wedge}|&
\le\frac{C_{\alpha}}{\pi}\int_0^{\infty}t^{m}(t+\rho)^{-1}\,dt\cdot \rho^{-|\alpha|}\\
&\le C'_{\alpha}\rho^{m-|\alpha|},
\end{split}\end{equation*}
which proves our case.
\end{proof}

\begin{lemma}
Let $K$ be a distribution on $\Ge$ smooth away from the origin and satisfying the estimates
\begin{equation}\label{flag_estimates}
|D^{\alpha}K(x)|\le C_{\alpha}|x|^{m-Q-|\alpha|},
\qquad
x\neq0,
\end{equation}
for some $m>0$. Then,
$$
K=R+\nu,
$$
where $R\in S^{-m}(\Ge)$ and $\partial\nu\in L^1(\Ge)$ for every left-invariant differential operator on~$\Ge$. 
\end{lemma}

\begin{proof}
It is sufficient to observe that (\ref{flag_estimates}) implies that $\widehat{K}$ is smooth away from the origin and
\begin{equation*}
|D^{\alpha}\widehat{K}(\xi)|\le C_{\alpha}|\xi|^{-m-|\alpha|},
\qquad
\xi\neq0,
\end{equation*}
and let $R=\phi K$, $\nu=K-R$, where $\phi\in C_c^{\infty}(\Ge)$ is equal to $1$ in a neighbourhood of $0$.
\end{proof}

Recall that
$$
P^m=V_m+\mu,
$$
where $V_m\in S^m(\Ge)$ and $\partial\mu\in L^1(\Ge)$ for every left-invariant differential operator $\partial$ on $\Ge$.

\begin{proposition}\label{decomposition}
Let $m>0$. Then
$$
(P^m+\delta_0)^{-1}=R+\nu,
$$
where $R\in S^{-m}(\Ge)$ and $\partial\nu\in L^1(\Ge)$ for every left-invariant differential operator $\partial$ on~$\Ge$.
\end{proposition}

\begin{proof}
Since the kernel $P^m$ is maximal (see (\ref{max}) above), it follows (see Dziu\-ba{\'n}ski \cite{dziubanski}, Theorem 1.13) that the semigroup generated by $P^m$ consists of operators with the convolution kernels 
$$
h_t(x)=t^{-Q/m}h_1(t^{-1/m}x), \qquad t>0,
$$ 
which are smooth functions satisfying the estimates
$$
|D^{\alpha}h_t(x)|\le\frac{C_{\alpha}t}{(t^{1/m}+|x|)^{Q+m+|\alpha|}}, \qquad x\in\Ge.
$$
Therefore,
$$
(P^m+\delta_0)^{-1}(x)=\int_0^{\infty}e^{-t}h_t(x)\,dt,
$$
and consequently satisfies the estimates (\ref{symbolic_estimates}).
\end{proof}

We know that  there exists a constant $C>0$ such that
\begin{equation*}
C^{-1}\|f\star V_1\|\le\|f\star P\|+\|f\|\le C\|f\star V_1\|,
\end{equation*}
whence
\begin{equation}\label{vge}
\|f\star V_1^m\|\ge C_m\|f\|,
\qquad
f\in\Schw(\Ge),
\end{equation}
for $m>0$. Now we have much more.
\begin{corollary}\label{equival}
For every $m>0$ there exists a constant $C>0$ such that
\begin{equation}\label{compare}
C^{-1}\|f\star V_1^m\|\le\|f\star P^m\|+\|f\|\le C\|f\star V_1^m\|.
\end{equation}
\end{corollary}

\begin{proof}
In fact, we have
$$
V_1^m=V_1^m\star(P^m+\delta_0)^{-1}\star(P^m+\delta_0)=\Big(V_1^m\star R+V_1^m\star\nu\Big)\star(P^m+\delta_0),
$$
where $R$ and $\nu$ are as in Proposition \ref{decomposition}. Then $V_1^m\star R\in S^0(\Ge)$ and $V_1^m\star \nu\in L^1(\Ge)$ so
$$
\|f\star V_1^m\|\le C_1(\|f\star P^m\|+\|f\|).
$$
The proof of the opposite inequality uses the identity
$$
f\star P^m=f\star V_m\star V_1^{-m}\star V_1^m+f\star\mu
$$
and (\ref{vge}).
\end{proof}

\section{Main theorem}

Here comes our main theorem and the conclusion of its proof.

\begin{theorem}\label{ef}
Let $A\in S^m(\Ge)$, where $m\ge0$. If $A$ satisfies the estimate
$$
\|f\star A\|\ge C(\|f\star P^m\|+\|f\|),
\qquad
f\in\Schw(\Ge),
$$
then there exists $B\in S^{-m}(\Ge)$ such that
$$
B\star A=\delta_0
$$
\end{theorem}

\begin{proof}[Conclusion of proof]
Let $A\in S^{m}(\Ge)$ satisfy the hypothesis of our theorem. Then $A\star V_1^{-m}$ satisfies the hypothesis of Corollary \ref{fe} so  there exists $B_1\in S^{0}(\Ge)$ such that
$$
B_1\star A\star V_1^{-m}=\delta_0.  
$$
By acting by convolution with $V_1^m$ on the right and with $V_1^{-m}$ on the left, we see that $B=V_1^{-m}\star B_1$ is the left-inverse for $A$.
\end{proof}

\begin{corollary}\label{comp}
Let $A=A^{\star}\in S^m(\Ge)$ for some $m\ge0$. The following conditions are equivalent:
\begin{enumerate}\item
There exists $B\in S^{-m}$ such that $B\star A=A\star B=\delta_0$,

\item
For every $k\in\R$, $\op(A):H(k+m)\to H(k)$ is an isomorphism,

\item
$\op(A):H(m)\to H(0)$ is an isomorphism,

\item
There exists $C>0$ such that
$$
\|f\star A\|\ge C(\|f\star P^m\|+\|f\|),
\qquad
f\in\Schw(\Ge).
$$
\end{enumerate}
 \end{corollary}

\begin{corollary}\label{para}
Let $A\in S^m(\Ge)$, where $m>0$, and let $\op(A)$ be positive in $L^2(\Ge)$. Then $A$ has a parametrix if and only if there exists $C>0$ such that
\begin{equation}\label{last}
||f\star{A}\|+\|f\|\ge C\|f\star P^m\|,
\qquad
f\in\Schw(\Ge).
\end{equation}
\end{corollary}

\begin{proof}
Let $B\in S^{-m}(\Ge)$ be a parametrix for $A$. Then
$$
B\star A=\delta_0+h,
$$
where $h\in\Schw(\Ge)$. Consequently,
$$
P^m=V_1^m\star B\star A+ g
$$
where $g\in L^1(\Ge)$. Now, $V_1^m\star B\in S^{\,0}(\Ge)$  so it is easy to see that the estimate (\ref{last}) holds.

Suppose now that (\ref{last}) holds true. Since $A$ is positive this implies
$$
\|f\star P^m\|\le C_1\|f\star({A}+\delta_0)\|,
$$
which, by Corollary \ref{comp}, implies that $A+\delta_0\in S^{m}(\Ge)$ has an inverse $B_1\in S^{-m}$. Thus
$$
B_1\star A=\delta_0-B_1,
$$
and the parametrix $B$ can be found as an asymptotic series
$$
B\approx\sum_{k=1}^{\infty}B_1^k.
$$
\end{proof}

\bigskip
\section{Rockland operators}\label{rockland}

A left-invariant homogeneous differential operator $R$ is said to be a {\it Rockland operator} if for every nontrivial irreducible unitary representation $\pi$ of $\Ge$, $\pi_R$ is injective on the space of $C^{\infty}$-vectors of $\pi$. 

Let $R$ be a left-invariant differential operator homogeneous of degree $\lbd=-Q-m$, that is,
$$
R(f\circ\delta_t)=t^mRf,
\qquad
f\in\Schw(\Ge),\quad t>0.
$$
It is well-known that the following conditions are equivalent:

\begin{enumerate}
\item
$R$ is a Rockland operator,
\item
$R$ is hypoelliptic,
\item
For every regular kernel $T$ of order $m$, there exists a constant $C>0$ such that
$$
\|\op(T)f\|\le C\|Rf\|,
\qquad
f\in\Schw(\Ge).
$$
\end{enumerate}

That (1) is equivalent to (2) was proved by Helffer-Nourrigat \cite{helffer} with a contribution from Beals \cite{beals} and Rockland \cite{rockland}. Helffer-Nourrigat \cite{helffer} also contains the proof of equivalence of (1)-(3) for $\op(T)$ being a differential operator. The remaining part was obtained by the author in \cite{glowacki1} and \cite{glowacki3}. 

It has been proved by Melin \cite{melin} that a Rockland operator on a {\it stratified} homogenenous group has a parametrix. We are going to show that in fact this is so on any homogeneous group.

\begin{corollary}
A Rockland operator on $\Ge$ has a parametrix.
\end{corollary}

\begin{proof}
Without any loss of generality we may assume that $R$ is positive. Then the assertion follows from (3) and Corollary \ref{para}.
\end{proof}

Thus we have one more condition equivalent to (1)-(3). However, the techniques of the present paper can be applied directly to Rockland operators rending unnecessary any reference to Theorem \ref{ef} or Corollary \ref{para}. What is needed are well-known properties of Rockland operators and the symbolic calculus of Proposition \ref{calculus}. Here is a brief sketch of a direct parametrix construction for a Rockland operator $R$.

We may assume that $R$ is positive. By Folland-Stein \cite{folland}, Chapter 4.B, $R$ is essentially selfadjoint on $L^2(\Ge)$ with $\Schw(\Ge)$ for its core domain. Moreover, the semigroup generated by it consists of convolution operators with kernels 
$$
p_t(x)=t^{-Q/m}p_1(t^{-1/m}x),
$$
where $p_1$ is a Schwartz class function. Note that $R=\op(R\delta_0)$. Let $S=(\delta_0+R\delta_0)^{-1}$. It follows that
$$
\widehat{S}(\xi)=\int_0^{\infty}e^{-t}\widehat{p_1}(t^{1/m}\xi)\,dt
$$
is a smooth function satisfying the estimates which show that $S\in S^{-m}(\Ge)$. Moreover,
$$
S\star R\delta_0=\delta_0-S,
$$
and by the usual argument the asymptotic series
$$
S_1\approx\sum_{k=1}^{\infty}S^k
$$
defines a parametrix for $R$ (cf. Melin \cite{melin}).

\section{Acknowledgements}
The author is grateful to M. Christ and F. Ricci for their helpful remarks on the subject of the present paper. He also thanks the referee for his careful reading of the manuscript and pointing out editorial omissions.

%\affiliationone{
  % Pawe{\l} G{\l}owacki\\
   %Institute of Mathematics,\\
   %University of Wroc{\l}aw,\\
   %pl. Grunwaldzki 2/4,\\
   %50-384 Wroc{\l}aw, Poland,\\
   %\email{glowacki@math.uni.wroc.pl}
   %}
   \end{document}